\documentclass[letterpaper, 10 pt, conference]{ieeeconf}  

\IEEEoverridecommandlockouts                              

\usepackage{amsmath,mathrsfs,amsfonts,amssymb,graphicx,epsfig} 
\usepackage{cite,url,bm,balance,physics}
\usepackage{stmaryrd,mathtools}
\makeatletter
\let\NAT@parse\undefined
\makeatother

\usepackage[dvipsnames]{xcolor}
\usepackage{hyperref}
\usepackage{subcaption}
\usepackage{algorithm}
\usepackage{algpseudocode}

\definecolor{pinegreen}{rgb}{0.0, 0.47, 0.44}

\hypersetup{
    colorlinks=true,
    linkcolor=blue,
    filecolor=magenta,      
    urlcolor=blue,
    citecolor= teal,
    pdftitle={Overleaf Example},
    pdfpagemode=FullScreen,
    }

    \usepackage{etoolbox} 

\newtheorem{theorem}{Theorem}

\newtheorem{proposition}{Proposition}

\title{\LARGE \bf
Second-Order Mean-Field Schr\"odinger Bridges\\with Partial State Information
}

\author{Asmaa Eldesoukey, Md Zulfiqur Haider, Italo Napolitano, Yongxin Chen, and Abhishek Halder
\thanks{Asmaa Eldesoukey and Abhishek Halder are with the Department of Aerospace Engineering, Iowa State University, Ames, IA 50011, USA, {\tt\footnotesize{\{{asmaae},ahalder\}@iastate.edu}.}
Md Zulfiqur Haider is with the Department of Mathematics, Iowa State University, Ames, IA 50011, USA, {\tt\footnotesize{zulfiqur@iastate.edu}}.
Italo Napolitano is with the Modeling and Engineering Risk and Complexity Department, Scuola
Superiore Meridionale, via Mezzocannone 4, 80134, Naples, Italy,  {\tt\footnotesize{i.napolitano@ssmeridionale.it}}.
Yongxin Chen is with the School of Aerospace Engineering, Georgia Institute of Technology, Atlanta, GA 30332, USA, {\tt\footnotesize{ychen3148@gatech.edu}}. This research was partially supported by NSF awards 2111688, 2450377, 2450378.}
}

\begin{document}

\maketitle
\thispagestyle{plain}
\pagestyle{plain}

\begin{abstract}
Motivated by scenarios in control systems where decision making is informed by only partial information about the states, we introduce a novel formulation of the Schr\"odinger bridge problem for \emph{inertial interacting-agent} systems. Within our setting, the agent ensemble is steered to match partial specifications of the collective configuration. Our analysis pertains to a mean-field framework in which the collective dynamics are described by a nonlinear Vlasov--Fokker--Planck integro-PDE with \emph{nonlocal} interactions. We assume that only distributions over positions are specified at the endpoints of a finite time horizon, representing aggregate population snapshots, while the corresponding velocity degrees of freedom remain unspecified. We derive the associated optimality system, identifying the optimal controlled evolution as jointly informed by the initial and terminal specifications through forward and backward propagation in time. We also show that the partial endpoint specifications bring about distinct conditions on the optimal control and the consistent evolution. Subsequently, to compute solutions to the optimality system, we develop a nested fixed-point algorithm inspired by the classical Fortet--Sinkhorn iteration, but extended to handle the nonlinearities arising from nonlocal interactions and temporal coupling. 

\end{abstract}


\section{INTRODUCTION}
Interacting-agent systems are common in a broad spectrum of control engineering applications, ranging from autonomous swarms \cite{annurev2022,An10023955} to cyber-physical infrastructures \cite{deng9023559} and to synthetic active matter \cite{bauerle2020formation}. In such complex systems, emergent macroscopic phenomena arise from pairwise interactions among constituent agents \cite{coraggio2026controlling}. Indeed, interaction mechanisms can effect desired collective behaviors, yet they may not align with finite-horizon control objectives. For instance, repulsive interactions promote collision avoidance in robotic swarms; however, they can hinder the control requirement to assemble within a target region. Thus, a natural control problem is to steer an agent ensemble toward a desired aggregate configuration while accounting for the interaction mechanisms integral to the collective dynamics.

The aim of this work is to design collective control policies guided by \emph{partial}, rather than full statistical knowledge of system states. Such a perspective is motivated by scenarios in which comprehensive information about interacting populations, especially those of large scale, may not be available, and only partial information can be gleaned from aggregate observations. Specifically, we consider the problem of finite-horizon steering of \emph{interacting inertial} agents subject to stochastic disturbances, in which only the \emph{spatial} probability distributions are specified at the initial and terminal times. The corresponding velocity distributions are left unspecified and are to be determined as part of the controlled evolution. We focus on dynamics in the \emph{mean-field limit}, where the system evolves according to a nonlinear Vlasov--Fokker--Planck equation with nonlocal interactions \cite{guillin2021kinetic}.

Our formulation is rooted in the theory of Schr\"odinger bridges that traces back to E.~Schr\"odinger's articles \cite{schrodinger1931,schrodinger1932}. Concerned with particle systems, the classical Schr\"odinger bridge problem seeks the most likely paths that the particles traverse between given initial and final probability distributions. From a stochastic control standpoint, the Schr\"odinger bridge identifies this intermediate stochastic evolution via a minimum-energy update of the underlying system dynamics \cite{daipra1991new}. Specifically, the update is an added drift vector field that is viewed as the control input required to steer the particles between the specified marginals.

\noindent{\textbf{Related works.}} The Schr\"odinger bridge problem with spatial endpoint constraints for \emph{non-interacting} second-order dynamics has been studied in \cite{chen2019multi,chiarini2022entropic}. Mean-field Schr\"odinger bridges have also been considered for interacting populations with \emph{first-order} dynamics (i.e., kinematics) in \cite{backhoff2020mean}. Numerical methods for addressing these first-order mean-field bridge problems have been developed in \cite{chen2024density,eldesoukey2026generalized}. Minimum-energy control for \emph{finite collections} of second-order stochastic oscillators has been explored in \cite{nodozi2022schrodinger}, but the mean-field limit formulation has not been considered. For broad classes of stochastic interacting systems, the works \cite{fischer2014form,chaintron2024quasi} showed, with different degrees of generality, that the required cost to identify the most likely evolution is a relative entropy functional and provided the equivalent stochastic control cost.

However, those works did not address the endpoint-constrained Schr\"odinger bridge problem for interacting kinetic systems, a direction the present authors recently initiated for Cucker--Smale and Morse swarming models in \cite{eldesoukey2026schr}. The present work builds on this, extending the formulation to systems under \emph{general} interactions. 


\noindent{\textbf{Contributions.}} We formulate the second-order mean-field Schr\"odinger bridge problem with partial state knowledge at the temporal endpoints and analyze the structure of its solutions. We explain how the resulting optimal controller and the corresponding evolution exhibit features distinct from those of the standard, fully informed setting. We show that the controlled dynamics can be recovered from a nonlinear system of nonlocally and temporally coupled forward-backward integro-PDEs. We further develop an extension to the Fortet--Sinkhorn algorithm to calculate solutions to this system. 


\noindent{\textbf{Organization.}} In Section~\ref{sec:problem}, we lay out the main problem. In Section~\ref{sec:solution-structure}, we derive the corresponding solution structure from the optimality conditions dictated by the partial marginal constraints. In Section~\ref{sec:algorithm}, we propose a nested Fortet--Sinkhorn-type algorithm to compute solutions to the optimality conditions numerically. Before concluding, we present a numerical example in Section~\ref{sec:Ex}.

\section{Main Problem}\label{sec:problem}
\subsection{Setting}

We consider $N$ inertial particles (agents), each of which follows the It\^o diffusion
\begin{subequations}\label{eq:finite-interacting}
    \begin{align}
    {\rm d} X_t^{i} &= V_t^i \, {\rm d} t, \\
  \!\!  {\rm d}V_t^{i} &= \frac{1}{N} \sum_{j =1,j\neq i}^N \!\! \Psi(X_t^i -X_t^j, V_t^i -V_t^j) \, {\rm d} t + \sigma \, {\rm d} W_t^i,
\end{align}
\end{subequations}
where $X_t^i, V_t^i \in \mathbb{R}^d$ respectively denote the position and velocity of the $i$-th agent, $i \in \{1, \ldots, N\}$, and $t \in [0,T]$. The symbol $\Psi$ denotes a known function
 $$\Psi: \mathbb{R}^{2d} \to \mathbb{R}^d,$$
representing pairwise interactions among the agents. We assume that $\Psi$ is globally Lipschitz continuous, ensuring \cite[Assumption~A.1]{chaintron2024quasi} and \cite[Conditions~(L),(G)]{fischer2014form} on the drift.
In addition, $W_t^1, \ldots, W_t^N$ are independent $d$-dimensional Wiener processes, and, for simplicity, we take the noise intensity $\sigma$ to be a positive constant, uniform across agents.

One can describe the collective state of the ensemble of agents following \eqref{eq:finite-interacting} via the instantaneous empirical distribution
\begin{align*}
   p_t^N:=\frac{1}{N} \sum_{i=1}^N  \delta_{(X_t^i,V_t^i)},
\end{align*}
where $\delta_{(X_t^i,V_t^i)}$ is the Dirac measure centered at $(X_t^i,V_t^i)$.
In our exposition, we shall focus on interacting models as in \eqref{eq:finite-interacting} in the \emph{mean-field limit}, where the number of agents $N$ tends to infinity and interactions take place via averaged effects \cite{bolley2011stochastic}. 

Within the mean-field regime, the empirical measure converges to a deterministic phase-space distribution, say $p_t$, and the mean-field dynamics corresponding to \eqref{eq:finite-interacting} become
\begin{subequations}\label{eq:general-interacting kinetics}
    \begin{align}
      \begin{bmatrix}
      {\rm d} X_t \\
      {\rm d} V_t 
  \end{bmatrix} = \begin{bmatrix} V_t \\ b_t [p_t] \end{bmatrix}  \, {\rm d} t 
  + \begin{bmatrix} 
    \bm 0_d \\ \sigma I_d
\end{bmatrix}  {\rm d} W_t,\label{eq:general-interacting kinetics-1}
\end{align} 
where $\bm 0_d$ is the $d \times d$ zero matrix and $I_d$ is the $d$-dimensional identity matrix. In \eqref{eq:general-interacting kinetics-1}, $b_t[p_t]$ denotes the mean-field force that is given by 
\begin{align} 
    b_t[p_t](x,v)
    & =  \int_{\mathbb R^{2d}} \Psi(x- \tilde{x},v - \tilde{v}) p_t(\tilde{x},\tilde{v}) \, {\rm d} \tilde{x} \, {\rm d} \tilde{v}, \label{eq:convolution-expanded}
\end{align}
\end{subequations}
where, with a slight abuse of notation, $p_t$ also denotes the instantaneous probability density function (PDF) of the phase-space distribution.

The mean-field force $b_t[p_t]$ is called \emph{nonlocal} as it depends on the phase-space density of the collective. Notice also that the system in \eqref{eq:general-interacting kinetics} can be viewed as a McKean--Vlasov process \cite{PASCUCCI2026104978} on the phase space but with a degenerate noise intensity matrix.

In this work, we assume that information about the system described in \eqref{eq:general-interacting kinetics} is available as spatial snapshots of the particle ensemble. That is, we are given two PDFs $q_0$ and $q_T$, supported on subsets of $\mathbb R^d$, that represent the ensemble's spatial distribution at $t=0$ and $t=T$, respectively. These specifications amount to marginal probabilistic constraints: \begin{subequations} \label{eq:partial-marginals}
\begin{align}
  \int_{\mathbb R^d} p_0(x,v) \, {\rm d}v  = q_0 (x), \\
  \int_{\mathbb R^d} p_T(x,v) \, {\rm d}v  = q_T (x),
\end{align}
\end{subequations}
for almost every (a.e.) $x \in \mathbb R^d$.
Unlike the standard Schr\"odinger bridge problem, which entails full state distributions at the endpoints, the constraints in \eqref{eq:partial-marginals} provide only \emph{partial} information about the collective state, since the velocity degrees of freedom remain unspecified at $t\in \{0,T\}$.

Within the present setting, we consider a variation of the standard Schr\"odinger bridge problem in which the system dynamics in \eqref{eq:general-interacting kinetics} define a prior evolution of the particle ensemble that may not be consistent with the specifications in \eqref{eq:partial-marginals}. The problem objective is to determine the most likely update to the prior dynamics that is consistent with the spatial marginal constraints.

Specifically, let the space of sample paths of the process in \eqref{eq:general-interacting kinetics} be
\begin{align*}
    \Omega := \mathcal C([0,T]; \mathbb R^{2d}),
\end{align*}
i.e., the space of continuous functions on the indicated time window, and let $\mathcal M(\Omega)$ denote the set of probability measures on $\Omega$ (path measures). If $\Gamma\in\mathcal M(\Omega)$ denotes the path measure associated with the dynamics in \eqref{eq:general-interacting kinetics}, then our bridge problem seeks the most likely path measure among those in $\mathcal M(\Omega)$ that satisfy the marginal constraints in \eqref{eq:partial-marginals}.

However, a subtlety arises due to the nonlocal interactions: the probability density $p_t$ appears in the prior dynamics, and as a result, the prior path measure $\Gamma$ is not independent of any candidate update $P \in \mathcal M({\Omega})$ whose marginal flow coincides with $t \mapsto p_t$. Therefore, $\Gamma = \Gamma(P)$ (see, e.g., \cite{backhoff2020mean}). 

Under the Lipschitz condition on $\Psi$, and varying assumptions on the initial distribution of the ensemble, \cite[Theorem~2.9]{chaintron2024quasi} and \cite[Theorem~5.2]{fischer2014form} indicate that the relevant large-deviation rate function is given by the Kullback--Leibler divergence:
\begin{align}\label{eq:KL}
    {\rm KL}(P,\Gamma(P)) := \int_{\Omega} {\rm d} P \log \frac{{\rm d}P}{{\rm d} \Gamma(P)},
\end{align}
when $P$ is absolutely continuous with respect to $\Gamma(P)$ (a relation written as $P \ll \Gamma (P)$).
This rate function quantifies the exponential decay of the likelihood of empirical path measures approaching a law $P$. As such, a smaller rate value incurs a higher likelihood.
Thus, among all path measures compatible with \eqref{eq:partial-marginals}, the most likely one minimizes \eqref{eq:KL}. 

\subsection{Problem Formalism}
In line with the control formulation of Schr\"odinger bridges \cite{daipra1991new}, we can view the prior model in \eqref{eq:general-interacting kinetics} as representing uncontrolled dynamics. Updating the prior dynamics and, in turn, the prior measure $\Gamma$ can then be achieved via a control vector field that guides the ensemble between the spatial marginals $q_0$ and $q_T$. Within this context, the controlled (updated) mean-field system reads
    \begin{align} \label{eq:controlled-dynamics}
      \begin{bmatrix}
      {\rm d} X_t \\
      {\rm d} V_t 
  \end{bmatrix} = \begin{bmatrix} V_t \\ b_t [p_t] + u_t\end{bmatrix}  \, {\rm d} t 
  + \begin{bmatrix} 
    \bm 0_d \\ \sigma I_d
\end{bmatrix}  {\rm d} W_t,
\end{align}
where $u_t:\mathbb R^{2d}\to\mathbb R^d$ is the time-dependent control vector field, evaluated at the state as $u_t(X_t,V_t)$. Equivalently, we can describe the collective dynamics in \eqref{eq:controlled-dynamics} by the Vlasov--Fokker--Planck equation
\begin{align*}
     \partial_t p_t + \big \langle v &, \nabla_x p_t \big \rangle   \\
     &+\nabla_v  \cdot  \Big(p_t  \big(b_t[p_t] + u_t\big) \Big)  =  \frac{\sigma^2}{2}\Delta_v p_t,
\end{align*}
where $\langle \cdot, \cdot \rangle$ denotes the Euclidean inner product, $\nabla_x$ the gradient with respect to position, $\nabla_v \cdot (\cdot)$ the divergence with respect to velocity, and $\Delta_v$ the Laplacian with respect to velocity. 

By applying Girsanov's theorem \cite[Chapter~8]{oksendal2003stochastic} to \eqref{eq:controlled-dynamics} and \eqref{eq:general-interacting kinetics}, one can verify that  (cf. \cite[Lemma~B.1]{fischer2014form}, \cite[Lemma~B.1]{chaintron2024quasi})
\begin{align}\label{eq:KL-KE-equiv-2}
    {\rm KL}(P, &\Gamma(P)) =  {\rm KL}(p_0, \gamma_0) \nonumber\\
    & + \frac{1}{2\sigma^2} \int_0^T \!\!\! \int_{\mathbb R^{2d}} \Vert u_t(x,v)\Vert^2 p_t (x,v) \, {\rm d} x \, {\rm d} v \, {\rm d}t,
\end{align}
where $\gamma_0$ denotes the prior's initial phase-space distribution. Based on \eqref{eq:KL-KE-equiv-2}, we can state our Schr\"odinger bridge problem as follows.
\vspace{4pt}

Given a prior initial phase-space probability distribution $\gamma_0$ and two PDFs $q_0$ and $q_T$ supported on subsets of $\mathbb R^d$,
    \begin{subequations}\label{prob:main}
    \begin{align}
 \!\!\!   \!   \underset{(u_\cdot, p_\cdot) \text{ with } p_0 \ll \gamma_0}{{\text{minimize}}} & \   {\rm KL}(p_0,\gamma_0) + \frac{1}{2\sigma^2} \int_0^T \!\!\! \int_{\mathbb R^{2d}} \Vert u_t\Vert^2 p_t  \, {\rm d} x \, {\rm d} v \, {\rm d}t 
    \end{align}
    subject to
    \begin{align}
         \partial_t p_t + \big \langle v &, \nabla_x p_t \big \rangle  \nonumber \\
     &+\nabla_v  \cdot  \Big(p_t  \big(b_t[p_t] + u_t\big) \Big)  =  \frac{\sigma^2}{2}\Delta_v p_t, \label{eq:VFP-controlled} \\
    \int_{\mathbb R^d} & p_0(\cdot,v)  \, {\rm d}v  = q_0 (\cdot), \label{eq:BC1} \\
  \int_{\mathbb R^d} & p_T(\cdot,v) \, {\rm d}v  = q_T (\cdot). \label{eq:BC2}
    \end{align}
    \end{subequations}

\section{Structure of Solutions}\label{sec:solution-structure}

We note that Problem~\eqref{prob:main} is generally non-convex, and the task of establishing the existence of minimizers is nontrivial, which suggests this direction for future investigation. Here, we focus on identifying the structural properties of its solution. 
To this end, we examine the endpoint constraints \eqref{eq:BC1} and \eqref{eq:BC2} more closely.
Let ${\rm Dens}(\mathbb R^{m})$ represent the space of PDFs on $\mathbb R^{m}, m>0$, which are, by definition, integrable functions. For $q \in {\rm Dens}(\mathbb R^d)$, define the surface 
\begin{align}\label{eq:surface}
    &\mathcal S(q) =  \nonumber \\
    & \Big\{ p \in {\rm Dens}(\mathbb R^{2d})~ \Big\vert~ \!\!\!\int_{\mathbb R^{d}}  p(x,v) \, {\rm d} v = q (x)  \text{ for a.e. } x \in \mathbb R^d \Big\}.
\end{align}
Then,  Eqs.~\eqref{eq:BC1} and \eqref{eq:BC2} can be expressed as 
\begin{align} \label{eq:surface-p0-pf}
   p_0  \in  \mathcal S(q_0), \quad  p_T \in \mathcal S(q_T).
   \end{align}

Further, the Lagrangian of Problem~\eqref{prob:main} can be written as
\begin{align}
    \mathcal L &=  {\rm KL}(p_0,\gamma_0) +  \int_0^T \!\!\! \int_{\mathbb R^{2d}} \Big\{\frac{1}{2\sigma^2} \Vert u_t\Vert^2 p_t \nonumber \\
    & \quad + \Lambda_t \Big[ \partial_t p_t + \big \langle v , \nabla_x p_t \big \rangle   \nonumber \\
     & \quad+ \nabla_v  \cdot  \big( p_t ( b_t[p_t] +  u_t)  \big)  -  \frac{\sigma^2}{2}\Delta_v p_t\Big] \Big\} \, {\rm d} x \, {\rm d} v \, {\rm d}t,\label{eq:Lagrangian1}
\end{align} %
where $\Lambda_t= \Lambda_t(x,v)$ is a Lagrange multiplier (costate) that enforces the dynamics in \eqref{eq:VFP-controlled}, and $t \mapsto p_t$ is admissible if \eqref{eq:surface-p0-pf} holds. Throughout, we assume sufficient regularity of $u_\cdot, p_\cdot, \Lambda_\cdot$ that allows our following variational analysis. 

Using integration by parts, we rewrite \eqref{eq:Lagrangian1} as
\begin{align}\label{eq:Lagrangian2} 
    \mathcal L &=  {\rm KL}(p_0,\gamma_0) + \int_0^T \!\!\! \int_{\mathbb R^{2d}}  \Big\{ \frac{1}{2\sigma^2} \Vert u_t\Vert^2 p_t \nonumber \\
    & \quad - p_t \Big[  \partial_t \Lambda_t +  \big \langle v , \nabla_x \Lambda_t \big \rangle   \nonumber \\
     & \quad +  \big \langle   b_t[p_t] +  u_t, \nabla_v \Lambda_t  \big\rangle   +  \frac{\sigma^2}{2}\Delta_v \Lambda_t \Big] \Big\} \, {\rm d} x \, {\rm d} v \, {\rm d}t, \nonumber \\
     & \quad + \int_{\mathbb R^{2d}} \Big(p_T \, \Lambda_T - p_0 \, \Lambda_0 \Big) \,  {\rm d} x \, {\rm d} v,
\end{align} 
where other boundary terms vanish when we consider $p_t \to 0$ sufficiently fast as $ \Vert x \Vert, \Vert v \Vert \to \infty$.

Having introduced the constraint surface in \eqref{eq:surface} and the Lagrangian in \eqref{eq:Lagrangian2}, we are ready to state our first result.

\begin{proposition}[Endpoint Costate Structure]\label{proposition:endpoint-conditions} Let the minimizer of Problem~\eqref{prob:main} be denoted by $[0,T]\ni t \mapsto (u_t^\star, p_t^\star)$.
   Also let $\gamma_0({\rm d}x, {\rm d}v) = \gamma_0(x,v)\, {\rm d}x \, {\rm d}v$ and
   \begin{align} \label{eq:quanti1}
     \xi_0(x,v) := 1 - \Lambda_0(x,v) + \log \frac{p_0^\star(x,v)}{\gamma_0(x,v)},
   \end{align} 
   for $p_0^\star(x,v) >0$.
Then, whenever the minimizer exists, 
   \begin{align} 
       \xi_0(x,v) &= \xi_0(x)  \label{eq:xi_0}
   \end{align} 
   for a.e. $(x,v) \in \mathbb R^{2d}$ at which $p_0^\star(x,v)>0$, and
   \begin{align}
       \Lambda_T(x,v) &= \Lambda_T(x), \label{eq:Lambda_T}
   \end{align} for a.e.  $(x,v) \in \mathbb R^{2d}$. That is, $\xi_0, \Lambda_T$ are independent of the velocity $v$ almost everywhere in their respective domains.
\end{proposition}
\begin{proof}
The endpoint boundary conditions \eqref{eq:BC1} and \eqref{eq:BC2} define the surfaces $\mathcal S(q_0)$ and $\mathcal S(q_T)$, as in \eqref{eq:surface}. Thus, at the minimizer $t \mapsto (u_t^\star, p_t^\star)$,  it is necessary that $p_0^\star \in \mathcal S(q_0)$ and $p_T^\star \in \mathcal S(q_T)$.

Any admissible variations for the optimal value $p_0^\star$ and $p_T^\star$, namely $\delta p_0^\star$ and $\delta p_T^\star$, must be tangent to $\mathcal S(q_0)$ and $\mathcal S(q_T)$; cf.~\cite[Section 5.1]{kirk2004optimal}.  That is,
 \begin{align}\label{eq:admissible-variations}
     \delta p_0^\star \in T_{p_0^\star} (\mathcal S(q_0)), \quad   \delta p_T^\star \in T_{p_T^\star} (\mathcal S(q_T)).
 \end{align}
 This means
 \begin{align} \label{eq:zero-integral-variations}
     \int_{\mathbb R^{d}} \delta p_0^\star  \, {\rm d} v =   \int_{\mathbb R^{d}} \delta p_T^\star \, {\rm d} v =0.
 \end{align}
 
Let us define the projection mapping
 \begin{align}
    &\pi: L^1(\mathbb R^{2d}) \to L^1(\mathbb R^d), \nonumber \\
 \text{where }  \quad  &\pi (p)(x) = \int_{\mathbb R^{d}} p(x,v) \, {\rm d} v, \label{eq:pi-definition}
\end{align}
which is a bounded linear operator. 
By \eqref{eq:zero-integral-variations}, we have
\begin{align}
    \delta p_0^\star, \delta p_T^\star  \in \mathcal N (\pi),
\end{align}
where $\mathcal N(\pi)$ denotes the null space of $\pi$.

The first-order necessary conditions of optimality for $p_0^\star,p_T^\star$ are defined by the vanishing variation of $\mathcal L$ in \eqref{eq:Lagrangian2} with respect to $p_0,p_T$. These are
\begin{subequations}\label{eq:orthogonality}
\begin{align}
    \int_{\mathbb R^{2d}} \xi_0(x,v) \delta p_0^\star(x,v) \, {\rm d}x \, {\rm d}v =0, \\
  \int_{\mathbb R^{2d}}  \Lambda_T(x,v)  \delta p_T^\star(x,v) \, {\rm d}x \, {\rm d }v =0,
\end{align}
\end{subequations}
where $\xi_0$ is given in \eqref{eq:quanti1} and $\delta p_0^\star$ and $\delta p_T^\star$ are as in \eqref{eq:admissible-variations}. Consequently, the quantities $\xi_0, \Lambda_T$ are \emph{orthogonal} to the admissible variations $ \delta p_0^\star,  \delta p_T^\star$ in the sense of \eqref{eq:orthogonality}.
This means
\begin{align}
    \xi_0 , \Lambda_T \in \overline{\mathcal R(\pi^\dagger)},
\end{align}
where $\overline{\mathcal R(\pi^\dagger)}$ denotes the closure of the range of the adjoint of $\pi$. This adjoint is defined as 
\begin{align}
    \int_{\mathbb R^d} (\pi f)(x) g(x) \, {\rm d}x  = \int_{\mathbb R^{2d}} f (x,v) (\pi^\dagger g)(x,v) \, {\rm d}x \, {\rm d}v, 
\end{align}
for all $f \in L^1(\mathbb R^{2d})$ and $g \in (L^1(\mathbb R^{d}))^\dagger=L^\infty(\mathbb R^{d})$. Equivalently, by \eqref{eq:pi-definition}, 
    \begin{align*}
       \int_{\mathbb R^{2d}} f (x,v) &  \, g(x) \, {\rm d} v \, {\rm d} x \nonumber \\
       &=   \int_{\mathbb R^{2d}} f (x,v) (\pi ^\dagger g)(x,v) \, {\rm d} v  \,  \, {\rm d} x.
    \end{align*}
Since $f$ is arbitrary, the previous equation implies 
\begin{align}
    (\pi^\dagger g)(x,v) = g(x),
\end{align}
and $\mathcal R(\pi^\dagger)$ is the set of bounded measurable functions that are independent of the velocity coordinate.
Then, indeed, $\mathcal R(\pi^\dagger)= \overline{\mathcal R(\pi^\dagger)}$. This completes the proof.
\end{proof}

In essence, the preceding proof relies on the geometry of the endpoint constraints to identify admissible variations of the ensemble density as directions tangent to the corresponding constraint surfaces at the temporal endpoints. Building on this identification and optimal control theory, we deduce that the terminal costate $\Lambda_T$ and the initial-costate-dependent quantity $\xi_0$ lie in the corresponding normal spaces, which implies their independence from the velocity coordinates. Next, we present the structure of the optimal controller for Problem~\eqref{prob:main}. 

\begin{proposition}[Optimal Control Structure]\label{proposition:controller}
Any minimizer $[0,T] \ni t \mapsto (u_t^\star, p_t^\star)$ of Problem~\eqref{prob:main}, whenever it exists, must satisfy
   \begin{align}
   u_t^\star(x,v) &= \sigma^2 \nabla_v \Lambda_t (x,v),  \label{eq:u-interacting-Lambda}
   \end{align} 
   for a.e. $(t,x,v) \in [0,T] \times \mathbb R^{2d}$, at which $p_t^\star(x,v)>0$. Additionally,
   \begin{align} \label{eq:u_T}
     u_T^\star (x,v) =0, 
   \end{align}
   for a.e. $(x,v) \in \mathbb R^{2d}$ at which $p_T^\star(x,v)>0$.
\end{proposition}
\begin{proof}
    The first variation of $\mathcal L$ in \eqref{eq:Lagrangian2} with respect to $u_\cdot$ is given by
\begin{align*}
 \!\!\! \!  \delta \mathcal L (u_\cdot, \delta u_\cdot) = \int_0^T \!\!\!  \int_{\mathbb R^{2d}}  p_t \Big \langle \Big(\frac{u_t}{\sigma^2} - \nabla_v \Lambda_t\Big) , \delta u_t  \Big \rangle  \, {\rm d} x \, {\rm d}  v  \, {\rm d}t.
\end{align*}
At the minimizer, $\delta \mathcal L (u_\cdot^\star, \delta u_\cdot^\star) =0$, which implies \eqref{eq:u-interacting-Lambda}.
Furthermore, by \eqref{eq:u-interacting-Lambda} and the terminal condition \eqref{eq:Lambda_T} in Proposition~\ref{proposition:endpoint-conditions}, we see that $\nabla_v \Lambda_T = 0$. Therefore, Eq.~\eqref{eq:u_T} follows.
\end{proof}

Proposition~\ref{proposition:controller} shows that the optimal control is given by a gradient vector field, much like the structure observed in classical kinematic bridges (see, e.g., \cite{daipra1991new}). However, the gradient is taken with respect to the velocity variable. As the terminal costate is independent of velocity, its velocity gradient vanishes, and hence so does the optimal control at the terminal time. This is a distinct feature of the minimizer for Problem~\eqref{prob:main}. 

\begin{theorem}[Mean-field Schr\"odinger Bridge] \label{thm:main}
   Whenever a minimizer $[0,T] \ni t \mapsto (u_t^\star, p_t^\star)$ of Problem~\eqref{prob:main} exists, there exist $\varphi_t: \mathbb R^{2d} \to \mathbb R_{++}, \hat \varphi_t: \mathbb R^{2d} \to \mathbb R_{+}$ such that
    \begin{subequations}
    \begin{align}
    u_t^\star(x,v) &= \sigma^2 \, \nabla_v \log \varphi_t(x,v), \label{eq:u-interacting-phi} 
\end{align}
for a.e. $(t,x,v) \in [0,T] \times \mathbb R^{2d}$ at which $p_t^\star(x,v)>0$, and
  \begin{align}
    u_T^\star &=0, \label{eq:u-interacting-T}
    \end{align}
    almost everywhere at which $p_T^\star>0$.
Additionally, 
\begin{align}
    p_t^\star (x,v) &= \varphi_t(x,v) \hat \varphi_t(x,v),  \label{eq:pt-phi-hatphi}
\end{align}
for a.e. $(t,x,v) \in [0,T] \times \mathbb R^{2d}$.
\end{subequations}
The pair $(\varphi_t, \hat \varphi_t)$ satisfies  
    \begin{subequations}\label{eq:Schr-sys-main}
        \begin{align} 
    \partial_t \varphi_t + \big \langle v ,  \nabla_x \varphi_t \big \rangle   &+ \big\langle b_t[p_t^\star] , \nabla_v\varphi_t  \big \rangle   \nonumber \\
    &+\frac{\sigma^2}{2} \Delta_v \varphi_t = h_t[p_t^\star] \varphi_t,  \label{eq:varphit}  \\
     \partial_t \hat\varphi_t + \big \langle  v , \nabla_x \hat\varphi_t  \big \rangle  &+ \nabla_v  \cdot (\hat\varphi_t \, b_t[p_t^\star]) \nonumber \\
     &-\frac{\sigma^2}{2} \Delta_v \hat\varphi_t = -  h_t [p_t^\star] \hat\varphi_t, \label{eq:hatvarphit}
\end{align}
where
\begin{align*}
    \!\!\!\! h_t&[p_t^\star](x,v) = \nonumber \\
    & -\int_{\mathbb R^{2d}} \hat \varphi_t(\tilde{x},\tilde{v})  \Big \langle \Psi (\tilde{x}-x, \tilde{v} -v),   \nabla_{\tilde v} \varphi_t(\tilde{x},\tilde{v}) \Big \rangle \, {\rm d}\tilde{x}\: {\rm d}\tilde{v}.
\end{align*}
Eqs.~\eqref{eq:varphit} and \eqref{eq:hatvarphit} are subject to the following boundary conditions.
\begin{align} 
   &\hat \varphi_0(x,v) = \hat\eta_0(x) \gamma_0 (x,v),  \label{eq:BCs-partial-all-1} \\
&\hat \eta_0(x) \int_{\mathbb R^d} \gamma_0(x,v)\varphi_0(x,v)  \, {\rm d} v  = q_0(x),  \label{eq:BCs-partial-all-2} \\
   & \varphi_T(x,v) =\eta_T(x),  \label{eq:BCs-partial-all-3} \\
    &\eta_T(x) \int_{\mathbb R^d}  \hat \varphi_T(x,v) \, {\rm d} v = q_T(x),  \label{eq:BCs-partial-all-4}
\end{align}
for some measurable functions $\hat \eta_0 : \mathbb R^{d} \to \mathbb R_{+}, \eta_T: \mathbb R^{d} \to \mathbb R_{++}$, and a.e. $x,v \in \mathbb R^{d}$.
    \end{subequations}
\end{theorem}

\begin{proof}
We addressed optimization over $p_0,p_T$ in Proposition~\ref{proposition:endpoint-conditions} and the optimization over $u_\cdot$ in Proposition~\ref{proposition:controller}. It remains to address optimization over $p_t$, where $t\in (0,T)$. To this end, we proceed by substituting \eqref{eq:u-interacting-Lambda} into \eqref{eq:Lagrangian2}. This yields
\begin{align*}
     \mathcal L &=   {\rm KL}(p_0,\gamma_0) + \int_0^T \!\!\!  \int_{\mathbb R^{2d}} p_t\Big\{ -\partial_t \Lambda_t - \big \langle v , \nabla_x \Lambda_t \big \rangle  \\
     &- \big \langle b_t[p_t], \nabla_v \Lambda_t \big \rangle    -   \frac{\sigma^2}{2} \Vert \nabla_v\Lambda_t  \Vert^2-  \frac{\sigma^2}{2}\Delta_v \Lambda_t \Big\}\, {\rm d} x\: {\rm d} v\: {\rm d}t \nonumber \\
     &+  \int_{\mathbb R^{2d}} \Big(p_T \, \Lambda_T - p_0 \, \Lambda_0 \Big) \,  {\rm d} x \, {\rm d} v.  
\end{align*}
Letting
\begin{align}
     \Upsilon :=  -  \int_0^T \!\!\!\int_{\mathbb R^{2d}} p_t  \big \langle  b_t[p_t], \nabla_v \Lambda_t \big\rangle \, {\rm d} x\: {\rm d} v\: {\rm d}t, \label{eq:upsilon-cs}
\end{align}
the first variation with respect to $p_t$ for $t \in (0,T)$ reads
\begin{align} \label{eq:dLdmu}
\!\!\!    \delta \mathcal L (p_\cdot, &\delta p_\cdot) = \int_0^T \!\!\!  \int_{\mathbb R^{2d}}   \Big\{ -\partial_t \Lambda_t - \big \langle v , \nabla_x \Lambda_t \big \rangle \nonumber \\
     & \qquad -  \frac{\sigma^2}{2} \Vert \nabla_v \Lambda_t  \Vert^2-  \frac{\sigma^2}{2}\Delta_v \Lambda_t \Big\} \delta p_t \, {\rm d} x\: {\rm d} v\: {\rm d}t  \nonumber \\
     & \qquad+ \delta  \Upsilon(p_\cdot, \delta p_\cdot).
\end{align}

Direct computation using \eqref{eq:convolution-expanded} yields 
    \begin{align}\label{eq:intermediate-step-in-optimality}
\!\!\! \delta  \Upsilon(p_\cdot, \delta p_\cdot)  = & \int_0^T \!\!\!  \int_{\mathbb R^{2d}}  \bigg [-\big\langle b_t[p_t]  , \nabla_v \Lambda_t \big \rangle  \nonumber\\
&\qquad\quad+ h_t[p_t](x,v) \bigg]  \delta p_t \, {\rm d} x \:{\rm d} v \, {\rm d} t,
\end{align}
where
\begin{align}\label{eq:qt}
 \!\!\!   h_t&[p_t](x,v) : = \nonumber \\
    & -\int_{\mathbb R^{2d}} p_t(\tilde{x},\tilde{v})  \Big \langle \Psi (\tilde{x}-x, \tilde{v}-v),   \nabla_{\tilde v} \Lambda_t(\tilde{x},\tilde{v}) \Big \rangle \, {\rm d}\tilde{x}\: {\rm d}\tilde{v}.
\end{align}
By substituting \eqref{eq:intermediate-step-in-optimality} into \eqref{eq:dLdmu}, the vanishing first variation of the Lagrangian $\mathcal L$ with respect to $p_t$ with $t\in (0,T)$ yields the Hamilton--Jacobi--Bellman equation
\begin{align}\label{eq:HJB-psi} 
\partial_t \Lambda_t + \big \langle v &, \nabla_x \Lambda_t \big \rangle + \big\langle b_t[p_t^\star] ,\nabla_v \Lambda_t  \big \rangle  \nonumber \\
 &+ \frac{\sigma^2}{2 }\Vert \nabla_v\Lambda_t  \Vert^2+ \frac{\sigma^2}{2}\Delta_v \Lambda_t =   h_t [p_t^\star].
\end{align}

Now, we define the change of variables 
\begin{align}\label{eq:log-transform}
    \Lambda_t(x,v) = \log \varphi_t(x,v),
\end{align}
where $\varphi_t$ is positive. This transformation is known as the \emph{Cole--Hopf} \cite{cole1951quasi,hopf1950partial} or \emph{Fleming logarithmic transformation} \cite{fleming2005logarithmic}. 
Using \eqref{eq:HJB-psi} and the identity
\begin{align}
    \Vert \nabla_v \log \varphi_t \Vert^2+ \Delta_v \log \varphi_t = \frac{\Delta_v \varphi_t}{\varphi_t},
\end{align}
we deduce that $\varphi_t$ in \eqref{eq:log-transform} satisfies \eqref{eq:varphit}.
Further, using \eqref{eq:u-interacting-Lambda} and \eqref{eq:log-transform}, the optimal control can be expressed as \eqref{eq:u-interacting-phi}.

Let us introduce
\begin{align} \label{eq:phi-hatphi}
    \hat \varphi_t(x,v) := \frac{p_t^\star(x,v)}{\varphi_t(x,v)}, \qquad t\in [0,T],
\end{align}
which is equivalent to \eqref{eq:pt-phi-hatphi}.
Then, by Proposition~\ref{proposition:controller}, \eqref{eq:VFP-controlled}, and \eqref{eq:varphit}, Eq.~\eqref{eq:hatvarphit} follows.

By Eqs.~\eqref{eq:quanti1}, \eqref{eq:xi_0} in Proposition~\ref{proposition:endpoint-conditions} together with \eqref{eq:BC1}, \eqref{eq:log-transform}, and \eqref{eq:pt-phi-hatphi}, 
we have
\begin{align*}
    \varphi_0 (x,v) \hat \varphi_0(x,v) = e^{\Lambda_0(x,v)} e^{-1+\xi_0(x)} \gamma_0(x,v).
\end{align*}
Let $\hat \eta_0 (x):= e^{-1+\xi_0(x)}$, whenever $q_0>0$ and zero otherwise. Then, using the transformation in \eqref{eq:log-transform} yields \eqref{eq:BCs-partial-all-1}. Then, by \eqref{eq:phi-hatphi} and \eqref{eq:BC1}, we obtain \eqref{eq:BCs-partial-all-2}. 
Analogously, by Eq.~\eqref{eq:Lambda_T} in Proposition~\ref{proposition:endpoint-conditions} and \eqref{eq:log-transform}, we have
\begin{align} \label{eq:varphi_T}
     \varphi_T (x,v) = e^{\Lambda_T(x)}.
\end{align}
Letting $\eta_T (x):= e^{\Lambda_T(x)}$ yields \eqref{eq:BCs-partial-all-3}. By \eqref{eq:phi-hatphi} and \eqref{eq:BC2}, we recover \eqref{eq:BCs-partial-all-4}.  Lastly, \eqref{eq:u-interacting-phi} and \eqref{eq:varphi_T} imply \eqref{eq:u-interacting-T}, which is consistent with Proposition~\ref{proposition:controller}.  This concludes the proof.
\end{proof}

Theorem~\ref{thm:main} indicates that the running probability density of the interacting agents is influenced by information from both forward and backward time directions. In that, the density factors into a forward-diffusing quantity $\hat \varphi_t$ and a backward-diffusing one $\varphi_t$, both determined by an interacting analogue of the classical Schr\"odinger system \cite{schrodinger1931}, \cite[Section~5]{chetrite2021schrodinger}, the system of equations in \eqref{eq:Schr-sys-main}. This system is nonlinear and coupled, both nonlocally and in time. Eqs.~\eqref{eq:varphit} and \eqref{eq:hatvarphit} are integro-PDEs that depend nonlocally on $p_t^\star$, which itself depends on the pair $(\varphi_t, \hat \varphi_t)$. In addition to coupling at the temporal boundaries via \eqref{eq:BCs-partial-all-1}--\eqref{eq:BCs-partial-all-4}, the reaction rate $h_t[p_t^\star]$ in \eqref{eq:varphit}, \eqref{eq:hatvarphit}  depends explicitly on $\hat \varphi_t$ and $\nabla_{\tilde{v}} \varphi_t$. Consequently, the system is coupled throughout the interval $[0,T]$, and the quantities $\varphi_t, \hat \varphi_t, p_t^\star$ cannot be solved for independently. 

In the next section, we propose an algorithm that alternates between computing these quantities, while mitigating the involved nonlinearities. The algorithm can be viewed as a generalization of the well-known Fortet-Sinkhorn algorithm that is used to solve the classical Schr\"odinger system \cite{fortet1940resolution,sinkhorn1964relationship}.

\section{Computational Approach} \label{sec:algorithm}
Our proposed computational approach is presented in Algorithm~\ref{alg-full-partial}, comprising nested fixed-point iterations (indexed by $j,k$).
The initial guesses $p_\cdot^{(0)}$, $\varphi_\cdot^{(0)}$, and $\hat{\varphi}_\cdot^{(0)}$ may be taken as the solution to the corresponding non-interacting second-order Schr\"odinger bridge problem, obtained by setting the mean-field force and reaction rate in \eqref{eq:Schr-sys-main} to zero. In this case, the resulting terminal multiplier is automatically $\eta_T^{[0]}(x) =\varphi_T^{(0)}(x,\cdot)$.

In Algorithm~\ref{alg-full-partial}, the inner iteration, indexed by $j$, solves \eqref{eq:Schr-sys-main} with $p_t = p_t^{(k)}$ being fixed, as well as a reaction rate determined via delayed evaluation. Consequently, this step involves propagating the solution of the backward-forward linear Kolmogorov equations in time. 
Upon completion of the inner iteration, the outer loop updates $t \mapsto p_t^{(k)}$ according to \eqref{eq:iteration-nl2}, via a damped evaluation if found necessary.

The convergence is taken in reference to the Hilbert projective metric $d_H(\cdot, \cdot)$ \cite{bushell1973hilbert,georgiou2015positive}. The metric between two measurable functions $f,g: \mathbb R^{2d} \to \mathbb R_{+}$ that share the same essential support $\Xi$ is defined as
\begin{align*}
d_H(f,g) := \log \sup_{(x,v) \in  \Xi}\frac{f(x,v)}{g(x,v)}- \log \inf_{(x,v) \in \Xi} \frac{f(x,v)}{g(x,v)},
\end{align*} 
with the extremal bounds obtained in the essential sense. 
\begin{algorithm}
\caption{Second-Order Mean-Field Sinkhorn}
\label{alg-full-partial}
\begin{algorithmic}[1]
    \Require $\Psi, \sigma, T ,q_0$, $q_T$, $\gamma_0$, $\varphi_\cdot^{(0)}$ such that $\varphi_T^{(0)}$ is independent of $v$, $\hat\varphi_\cdot^{(0)}$, $p_\cdot^{(0)}$, $N_1,N_2$, \texttt{tol}, $\theta \in (0,1]$.
   \For{$k = 0, \, \ldots, N_1$}
        \State Set $\varphi_\cdot^{[0]} \leftarrow  \varphi_\cdot^{(k)}$, $\hat \varphi_\cdot^{[0]}  \leftarrow \hat \varphi_\cdot^{(k)}$, $\eta_T^{[0]}(x) \leftarrow \varphi_T^{(k)}(x, \cdot)$.
        \State Compute $b_\cdot[p_\cdot^{(k)}]$.
         \For{$j = 0, \, \ldots, N_2$}
    \State Compute  \begin{align*}
        &h_t^{[j]}(x,v) =  \nonumber \\
       &- \int_{\mathbb R^{2d}} \hat \varphi_t^{[j]}(\tilde{x},\tilde{v})  \Big \langle \Psi (\tilde{x}-{x}, \tilde{v} -v),   \nabla_{\tilde v} \varphi_t^{[j]}(\tilde{x},\tilde{v}) \Big \rangle \, {\rm d}\tilde{x}\: {\rm d}\tilde{v}.
    \end{align*}
     \State Integrate
    \begin{subequations}
     \begin{align}\label{eq:backward-alg-full-partial}
       \partial_t \varphi_t^{[j+1]}&= - \langle v , \nabla_x  \varphi_t^{[j+1]} \rangle- \langle b_t(p_t^{(k)}) , \nabla_v \varphi_t^{[j+1]} \rangle \nonumber \\
       &-\frac{\sigma^2}{2} \Delta_v  \varphi_t^{[j+1]} +h_t^{[j]}(x,v)  \varphi_t^{[j+1]}
     \end{align}
    backward in time  from the terminal condition $\varphi_T^{[j+1]}(x,v) = \eta_T^{[j]}(x)$.
\State Compute
 \begin{align}
      \hat \eta_0^{[j+1]} (x) &=   \frac{q_0(x)} {\int \gamma_0(x,v)  \varphi_0^{[j+1]}(x,v) {\rm d}v}.
 \end{align}
\State Integrate
              \begin{align}\label{eq:forward-alg-full-partial}
          \partial_t \hat\varphi_t^{[j+1]}  &=- \langle v , \nabla_x \hat\varphi_t^{[j+1]} \rangle -  \nabla_v \cdot(b_t(p_t^{(k)})\, \hat\varphi_t^{[j+1]}) \nonumber \\
          &+\frac{\sigma^2}{2} \Delta_v \hat\varphi_t^{[j+1]} -h_t^{[j]}  \hat\varphi_t^{[j+1]}
    \end{align}
 forward in time from the initial condition $$\hat \varphi_0^{[j+1]}(x,v) = \hat \eta_0^{[j+1]}(x) \gamma_0(x,v).$$
 \State    Compute
     \begin{align}
           \eta_T^{[j+1]} (x) =  \frac{q_T(x)}{\int\hat\varphi_T^{[j+1]}\, {\rm d}v}.
     \end{align}
      \end{subequations}
      \State Set  $\varphi_T^{[j+1]}(x,\cdot)\leftarrow\eta_T^{[j+1]}(x)$.
  \If{
\begin{align*}
\sup_{t \in [0,T]} \max\Big\{   d_H\!\left( \varphi_t^{[j+1]}, \varphi_t^{[j]} \right), d_H \!\left( \hat\varphi_t^{[j+1]}, \hat\varphi_t^{[j]} \right) \Big\} < \texttt{tol}
\end{align*}
} 
     \State \textbf{break}
        \EndIf
    \EndFor
     \State Set
    $\varphi_\cdot^{(k+1)}\leftarrow\varphi_\cdot^{[j+1]}$ and $\hat\varphi_\cdot^{(k+1)}\leftarrow\hat\varphi_\cdot^{[j+1]}$.
 \State Update
    \begin{align}
    p_\cdot^{(k+1)} = \theta\varphi_\cdot^{(k+1)}\hat\varphi_\cdot^{(k+1)} +(1-\theta)p_\cdot^{(k)}.    \label{eq:iteration-nl2}
    \end{align}
    \If{$
   \sup_{t \in [0,T]} d_H\!  \left(p_t^{(k+1)},p_t^{(k)}\right) <\texttt{tol} $}
    \State \textbf{break}
    \EndIf
    \EndFor
    \State \Return $\varphi_\cdot^{(k+1)}$, $\hat\varphi_\cdot^{(k+1)}$, and $p_\cdot^{(k+1)}$.
      \end{algorithmic}
\end{algorithm}

\section{Numerical Example}\label{sec:Ex}
\begin{figure}[ht]
    \centering
     \begin{subfigure}[b]{0.49\textwidth}
        \centering
        \includegraphics[width=\textwidth]{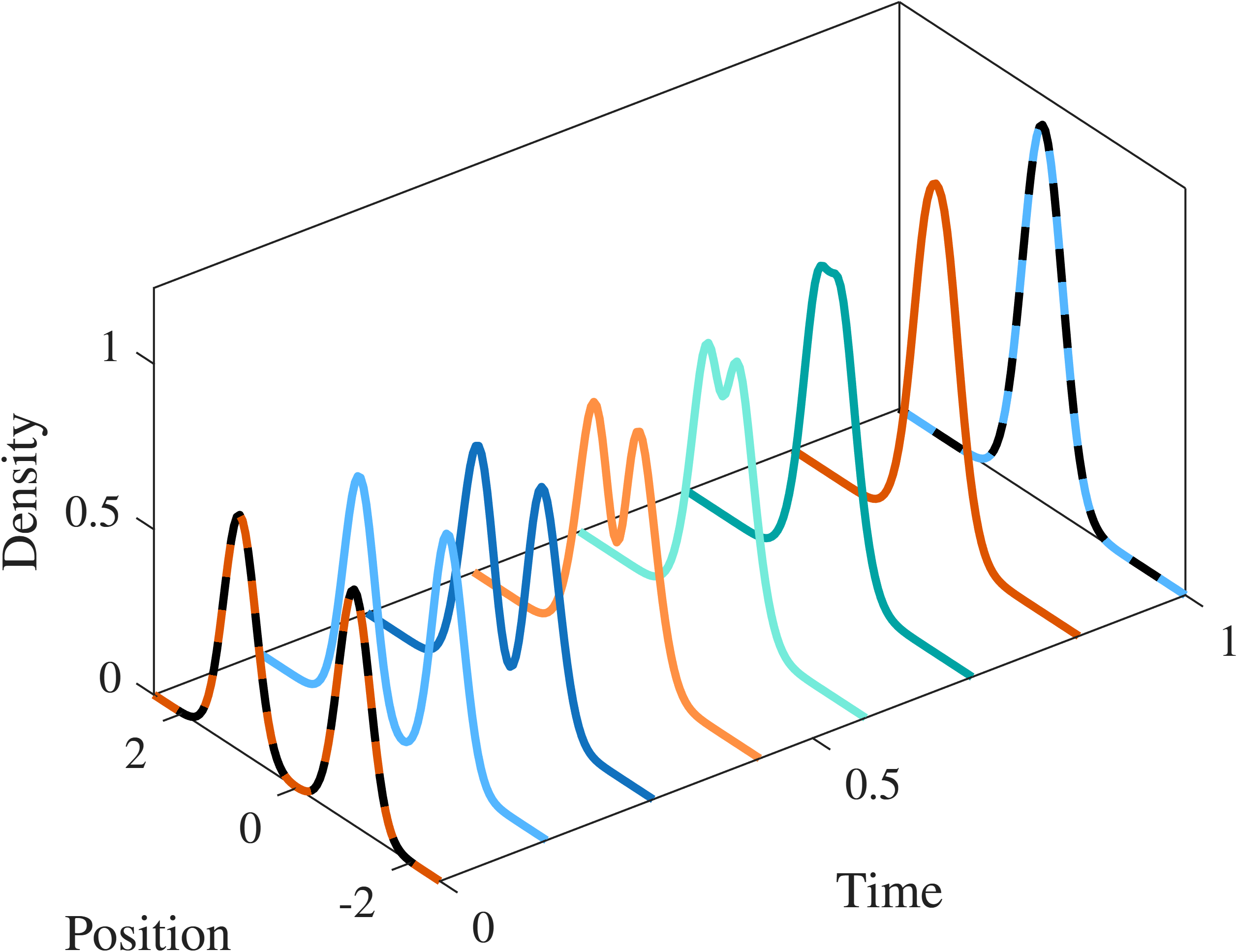}
         \vspace{10pt}
        \caption{Numerically computed spatial densities over time (solid curves) and the specified endpoint marginals $q_0$ and $q_T$ (dashed curves).}
        \label{fig:sub-first}
    \end{subfigure}
    \hfill 
    \vspace{10pt}
    \begin{subfigure}[b]{0.49\textwidth}
        \centering
        \includegraphics[width=\textwidth]{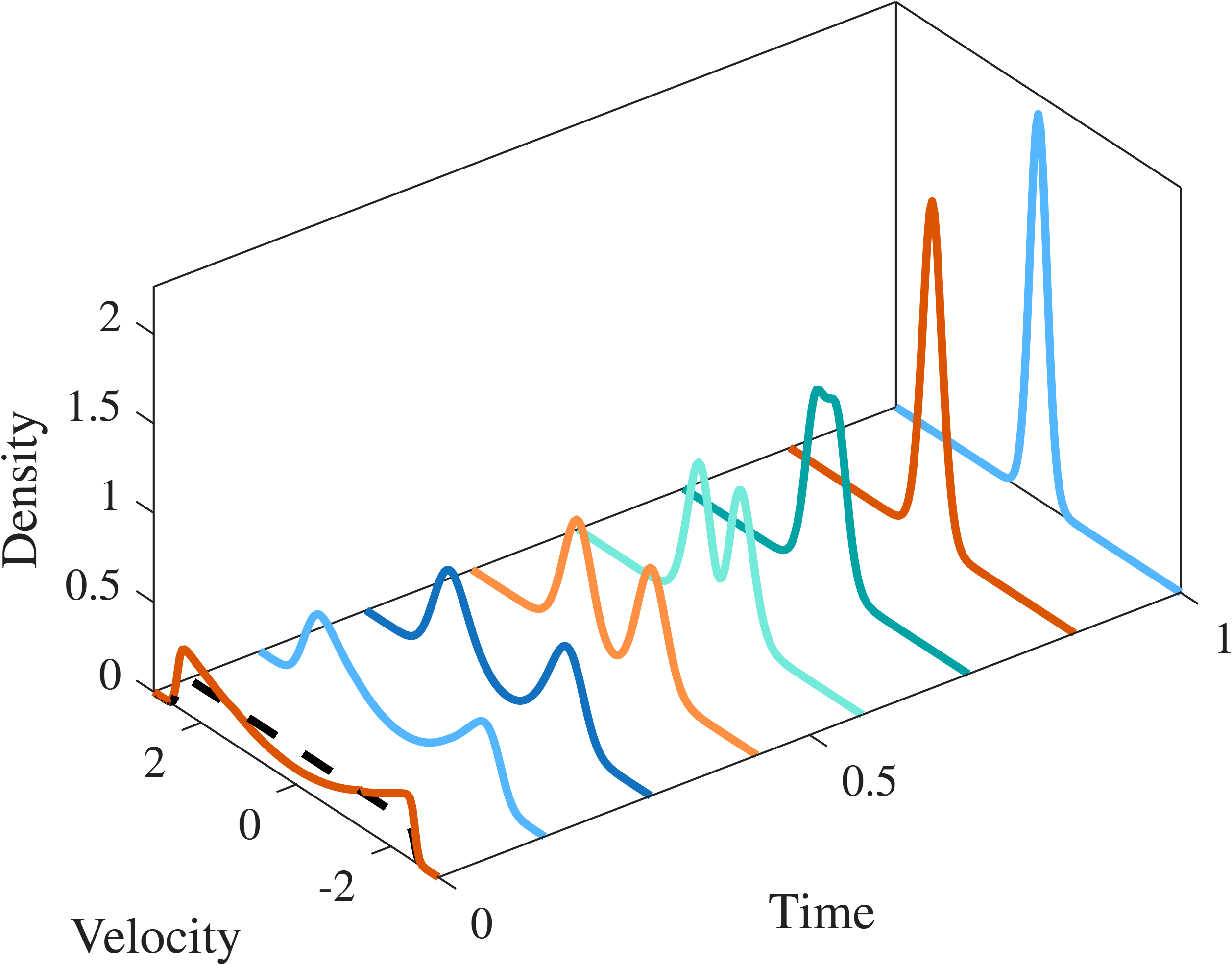}
         \vspace{10pt}
        \caption{Numerically computed velocity densities over time (solid curves) and the initial prior velocity marginal $\int \gamma_0 \, {\rm d}x$ (dashed curve).}
        \label{fig:sub-second}
    \end{subfigure}
    \caption{Results of the numerical example in Section~\ref{sec:Ex}.}
    \label{fig:main-cluster}
\end{figure}
Herein, we present an illustrative example of steering the ensemble dynamics given by \eqref{eq:general-interacting kinetics}, with $\sigma=0.5$, $T =1$, and
\begin{align*}
    \Psi(x,v) = 5 x \exp\Big({-\frac{x^2}{2(0.45)^2}}\Big) - 6 v \exp\Big({-\frac{x^2}{2(0.8)^2}} \Big),
\end{align*}
on the computational domain $x \in [-3,3] , v \in [-6,6]$.
The first term in this interaction kernel models positional repulsion, pushing nearby agents apart, and the second term represents velocity alignment, discouraging large velocity differences among neighboring agents.

We consider the prespecified spatial PDFs 
\begin{align*}
    q_0(x) &\propto  \frac{1}{2} \exp \Big(- \frac{(x-1)^2}{2(0.28)^2} \Big) +  \frac{1}{2} \exp \Big(- \frac{(x+1)^2}{2(0.28)^2} \Big),\\
    q_T (x)  &\propto   \exp \Big(- \frac{x^2}{2(0.35)^2} \Big).
\end{align*}
These PDFs are depicted in Fig.~\ref{fig:sub-first} with dashed curves. Additionally, we define the initial prior phase space distribution as
\begin{align*}
    \gamma_0(x,v) \propto q_0(x) \Big[ \tanh \Big(\frac{v+2.5}{0.08}\Big) -  \tanh \Big(\frac{v-2.5}{0.08}\Big) \Big],
\end{align*}
representing an approximately uniform velocity distribution over the interval $[-2.5, 2.5]$. Fig.~\ref {fig:sub-second} shows the velocity marginal of $\gamma_0$ as a dashed curve. 

Applying Algorithm~\ref{alg-full-partial} with $\texttt{tol}= 1 \times 10^{-5}$ and $\theta=1$ (no damping) yields the time-resolved spatial and velocity densities displayed in Figs.~\ref{fig:sub-first} and \ref{fig:sub-second}. As demonstrated, the resulting spatial marginals indeed match the prescribed ones at the temporal endpoints. The obtained solution corresponds to velocity marginals that gradually concentrate, steering the position dynamics toward the target distribution.

\section{Concluding remarks}

 This work has relied on Schr\"odinger's dictum of bridges to synthesize suitable collective control policies when sensing of an interacting ensemble is limited to partial, aggregate information. Specifically, we considered assimilating spatial snapshot data and inertial mean-field dynamics into the bridge formulation. We have identified the structure of the optimal policy and the resulting controlled evolution and proposed a nested fixed-point approach to compute the corresponding optimality system. 

Future research directions include uncertainty in interaction mechanisms. Another direction is to extend the present framework to entail time-resolved distributions (a multi-marginal formulation). This is particularly relevant to scenarios where time series of population snapshots are available, such as in trajectory inference \cite{chizat2022trajectory,pmlr-v267-hong25f}.

\balance
\bibliographystyle{IEEEtran}
\bibliography{main.bib}

\end{document}